%% file: main.tex
\documentclass[A4paper, 12pt]{article} % use larger type; default would be 10pt
\usepackage[T1]{fontenc}
\usepackage[utf8]{inputenc} % set input encoding (not needed with XeLaTeX)
\usepackage[all,cmtip]{xy}

\usepackage[margin=0.8in, top=1in, bottom=1in, footskip=0.5in]{geometry}
\usepackage{graphicx} % support the \includegraphics command and options

\usepackage{booktabs} % for much better looking tables
\usepackage{array} % for better arrays (eg matrices) in maths
\usepackage{paralist} % very flexible & customisable lists (eg. enumerate/itemize, etc.)
\usepackage{verbatim} % adds environment for commenting out blocks of text & for better verbatim
\usepackage{subfig} % make it possible to include more than one captioned figure/table in a single float
\usepackage{mathtools}
\usepackage{amssymb}
\usepackage{amsmath}
\usepackage{amsthm}
\usepackage{mathrsfs}
\usepackage{enumitem}
\usepackage[colorlinks, allcolors=blue]{hyperref}
\usepackage{appendix}
\usepackage{stmaryrd}
\usepackage{hyperref}
\usepackage[bottom]{footmisc}
\hypersetup{colorlinks,linkcolor={blue},citecolor={blue},urlcolor={red}}

\usepackage{subfiles}

\usepackage{fancyhdr} % This should be set AFTER setting up the page geometry
\usepackage[tightpage,psfixbb]{preview}
\usepackage{sectsty}

\sectionfont{\normalfont}
\subsectionfont{\normalfont\itshape}
\subsubsectionfont{\normalfont\itshape}
\makeatletter

\usepackage[nottoc,notlof,notlot]{tocbibind} % Put the bibliography in the ToC
\usepackage[titles,subfigure]{tocloft} % Alter the style of the Table of Contents
\theoremstyle{sltheoremstyle}
\newtheorem{theorem}{Theorem}[section]
\newtheorem{lemma}[theorem]{Lemma}
\newtheorem{corollary}[theorem]{Corollary}
\newtheorem{proposition}[theorem]{Proposition}

\newtheorem{condition}[theorem]{Condition}

\newtheorem*{corollary-non}{Corollary}
\newtheorem*{lemma-non}{Lemma}
\newtheorem*{theorem-non}{Theorem}
\newtheorem*{proposition-non}{Proposition}
\newtheorem*{condition-non}{Condition}
\newtheorem*{conditions-non}{Conditions}

\theoremstyle{definition}
\newtheorem{notation}[theorem]{Notation}
\newtheorem{remark}[theorem]{Remark}

\newtheorem{definition}[theorem]{Definition}

\usepackage[backend=bibtex,style=alphabetic]{biblatex}

\input{./macros.tex}

\title{\textsc{Real moduli spaces and density of non-simple real abelian varieties}}
\author{\normalsize Olivier de Gaay Fortman\textcolor{blue}{$^{\ast}$}}
\date{} % Activate to display a given date or no date (if empty),
\begin{document}

\pagestyle{empty}
%\newgeometry{top=1in, bottom=0.5in}

\maketitle

\pagestyle{empty}
%\newgeometry{top=1in, bottom=0.5in}
%\maketitle

\abstract{\centering{\scriptsize{\noindent
For fixed $k< g$ and a family of polarized abelian varieties of dimension $g$ over ${\bb R}$, we give a criterion for the density in the parameter space of those abelian varieties over ${\bb R}$ containing a $k$-dimensional abelian subvariety over ${\bb R}$. As application, we prove density of such a set in the moduli space of polarized real abelian varieties of dimension $g$, and density of real algebraic curves mapping non-trivially to real $k$-dimensional abelian varieties in the moduli space of real algebraic curves as well as in the space of real plane curves. This extends to the real setting results by Colombo and Pirola \cite{Colombo1990}. We then consider the real locus of an algebraic stack over $\RR$, attaching a topological space to it. For a real moduli stack, this defines a real moduli space. We show that for $\ca M_g$ and $\ca A_g$, the real moduli spaces that arise in this way coincide with the moduli spaces of Gross-Harris \cite{grossharris} and Sepp\"al\"a-Silhol \cite{seppalasilhol2}.} 
}
}

\pagestyle{empty}

\section{Introduction}
\label{introduction}

Fix an integer $g \geq 1$. Let $\ca A$ and $B$ be complex manifolds, and let
\begin{equation} \label{family}
\psi:  \ca A \to B, \white s: B \to \ca A, \white E \in R^2\psi_*\bb Z
\end{equation} be a polarized holomorphic family of $g$-dimensional complex abelian varieties. The map $\psi$ is a proper holomorphic submersion, $s$ is a section of $\psi$, and $A_t: = \psi^{-1}(t) $ is a complex abelian variety of dimension $g$ with origin $s(t)$, polarized by $E_t \in H^2(A_t, \bb Z)$, for $t \in B$. 
\\
\\
Suppose moreover that $\psi$ admits a real structure in the following sense: $\ca A$ and $B$ are equipped with anti-holomorphic involutions $\tau$ and $\sigma$, commuting with $\psi$ and $s$ and compatible with the polarization, in the sense $\tau^*(E) = -E$. For example, this is true when $\psi$ is real algebraic, i.e. induced by a polarized abelian scheme over a smooth $\bb R$-scheme.
\\
\\
Let $B({\bb R})$ be the set of fixed points under the involution $\sigma: B \to B$. If $t \in B({\bb R})$, then $A_t$ is equipped with an anti-holomorphic involution $\tau$ preserving the group law. We shall not distinguish between the category of abelian varieties over ${\bb R}$ and the category of complex abelian varieties equipped with an anti-holomorphic involution preserving the group law. Thus, if $t \in B({\bb R})$, then $A_t$ is an abelian variety over $\bb R$. Define $R_k \subset B({\bb R})$:
\begin{equation}\label{rk}
R_k  = \left\{ t \in B({\bb R}): A_t \text{\textit{ contains an abelian subvariety over ${\bb R}$ of dimension }} k \right\}.
\end{equation}
For $t \in B$, the polarization gives an isomorphism $H^{0,1}(A_t) \cong H^{1,0}(A_t)^*$; using the dual of the differential of the period map we obtain a symmetric bilinear form 
\begin{equation} \label{symbil}
q: H^{1,0}(A_t) \otimes H^{1,0}(A_t) \to T_t^*B. 
\end{equation}

\blfootnote{\textcolor{blue}{$^\ast$}\footnotesize{\'Ecole normale sup\'erieure, 45 rue d'Ulm, Office T17, 75230 Paris, \href{mailto:olivier.de.gaay.fortman@ens.fr}{olivier.de.gaay.fortman@ens.fr}.}} 
\blfootnote{\textcolor{blue}{$^\ast$}\textit{Date:} \today $\;$
- This project has received funding from the European Union's Horizon 2020 research and innovation programme under the Marie Sk\l{}odowska-Curie grant agreement N\textsuperscript{\underline{o}} 754362.$\;$\img{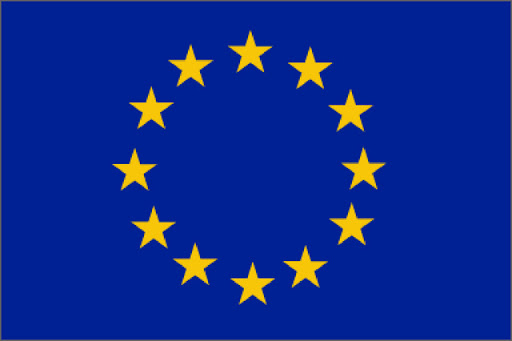}} 

\restoregeometry
\pagestyle{plain}

\begin{condition}  \label{criterion} 
There exists an element $t \in B$ and a $k$-dimensional complex subspace $W \subset H^{1,0}(A_t)$ such that the complex $0 \to  \bigwedge^2 W \to W \otimes H^{1,0}(A_t) \to T_t^*B $ is exact.
\end{condition}
\noindent
Our main theorem is the following.
\begin{theorem}  \label{theorem1} 
If $B$ is connected and if Condition \ref{criterion} holds, then $R_k$ is dense in $B({\bb R})$. 
\end{theorem}
\noindent
We give the following three applications of Theorem \ref{theorem1}. A \textit{real algebraic curve} will be a proper, smooth, geometrically connected curve over ${\bb R}$. Let $\mr M_g^{\bb R}$ be the set of isomorphism classes of real algebraic curves of genus $g$, and $\mr A_g^{{\bb R}}$ the set of isomorphism classes of real principally polarized abelian varieties of dimension $g$. The sets $\mr A_g^{{\bb R}}$ and $\mr M_g^{\bb R}$ carry natural real semi-analytic structures by work of Gross-Harris \cite{grossharris} and Sepp\"{a}l\"{a}-Silhol \cite{seppalasilhol2}.

\begin{theorem} \label{theorem2}
\begin{enumerate}
\item[\hypertarget{theoremA}{\textbf{A.}}] 
Given an integer $k$ with $1 \leq k \leq g-1$, abelian varieties over $\bb R$ \\
containing a $k$-dimensional abelian subvariety over ${\bb R}$ are dense in the moduli space \\
$\mr A_g^{{\bb R}}$ of principally polarized abelian varieties of dimension $g$ over ${\bb R}$. \item[\hypertarget{theoremB}{\textbf{B.}}] For $g \geq 3$ and $k \in \{1,2,3\}$, real algebraic curves $C$ that admit a map $\varphi: C \to A$ \\ with $A$ a $k$-dimensional abelian variety over ${\bb R}$ such that $\varphi(C)$ generates $A$ as an \\ algebraic group are dense in the moduli space $\mr M_g^{{\bb R}}$ of real genus $g$ algebraic curves.
\item[\hypertarget{theoremC}{\textbf{C.}}] If $V \subset \bb PH^0(\bb P^2_{{\bb R}}, \OO_{\bb P^2_{{\bb R}}}(d))$ is the real algebraic set of degree $d$ smooth plane curves \\ over ${\bb R}$, then the subset of $V$ corresponding to those curves that map non-trivially to \\ elliptic curves over $\bb R$ is dense in $V$. 
\end{enumerate}
\end{theorem}
 \noindent
 \begin{remark}
The topology in the moduli spaces of Theorem \ref{theorem2}.\hyperlink{theoremA}{A} and \ref{theorem2}.\hyperlink{theoremB}{B} is the one underlying the real semi-analytic structure. We prove in Section \ref{realan} that these topologies have other more intrinsic incarnations. It might also be worth noting that, although well-known in the complex case, Theorem \ref{theorem2}.\hyperlink{theoremA}{A} is new in the real case. 
\end{remark}
\noindent Our proofs rely on results in the complex setting that were proved by Colombo and Pirola in \cite{Colombo1990}. Indeed, Theorem \ref{theorem1} is the analogue over $\bb R$ (with \textit{unchanged} hypothesis) of the following theorem. Define $S_k \subset B$ to be the set of those $t$ in $B$ for which the complex abelian variety $A_t$ contains a complex abelian subvariety of dimension $k$.
\begin{theorem}[Colombo-Pirola \cite{Colombo1990}] \label{colpol}
If $B$ is connected and if Condition \ref{criterion} holds, then $S_k$ is dense in $B$. $\hfill \qed$
\end{theorem}
\noindent
Colombo and Pirola in turn were inspired by the Green-Voisin Density Criterion \cite[Proposition 17.20]{voisin}. Indeed, the latter gives a criterion for density of the locus where the fiber contains many Hodge classes for a variation of Hodge structure of weight $2$. Theorem \ref{colpol} adapts this result to a polarized variation of Hodge structure of weight $1$ (which is nothing but a polarized family of complex abelian varieties). The result is a criterion for the density of the locus where the fiber admits a sub-Hodge structure of dimension $k$. 
\\
\\
To be a little more precise, recall that for a complex manifold $U$ and a rational weight $2$ variation of Hodge structure $(H_{\bb Q}^2, \ca H, F^1, \nabla)$ on $U$, the \textit{Noether-Lefschetz locus} $\tn{NL}(U) \subset U$ is the locus where the rank of the vector space of Hodge classes is bigger than the general value. If $H_{\bb Q}^2$ is polarizable then $\tn{NL}(U)$ is a countable union of closed algebraic subvarieties of $U$ \cite{MR1273413}. The Green-Voisin Density Criterion referred to above decides whether $\tn{NL}(U)$ is dense in $U$. It was first stated in \cite{NLlocus} and applied to the universal degree $d\geq4$ surface $\mr S \to \mr B$ in $\bb P^3$. In this case, $\tn{NL}(\mr B)$ is the locus where the Picard group is not generated by a hyperplane section, and the union of the general components of $\tn{NL}(\mr B)$ is dense in $\mr B$ [\textit{loc. cit.}]. Analogously, $S_k$ is a countable union of components of $\NL(B)$ \cite{laszlodebarre}, and Theorem \ref{colpol} says that this union is dense in $B$. 
\\
\\
Now let us carry the discussion over to the real setting. Unfortunately, the Green-Voisin Density Criterion cannot be adapted to the reals without altering the hypothesis. Going back to the universal family $\mr S \to \mr B$ of degree $d\geq 4$ surfaces in $\bb P^3_{{\bb C}}$, one observes that this family has a real structure, so that we can define the \textit{real Noether-Lefschetz locus} $\NL(\mr B({\bb R})) \subset \mr B({\bb R})$ as the locus of real surfaces $S$ in $\bb P^3_{{\bb R}}$ with $\Pic(S) \neq {\bb Z}$. By the above, the Green-Voisin Density Criterion is fulfilled hence $\NL(\mr B)$ is dense in $\mr B$, whereas density of $\NL(\mr B({\bb R}))$ in $\mr B(\bb R)$ may fail: for every degree $4$ surface in $\bb P^3_{{\bb R}}$ whose real locus is a union of $10$ spheres, $\Pic(S) = {\bb Z}$, and so $\NL(\mr B({\bb R}))  \cap K = \emptyset$ for any connected component $K$ of surfaces of such a topological type \cite[Rem.1.5]{benoistttt}. There is an alternate criterion \cite[Prop.1.1]{benoistttt}, but the hypothesis is more complicated thus harder to fulfill, and only implies density of $\NL(\mr B({\bb R}))$ in one component of $\mr B(\bb R)$ at a time. It is therefore remarkable that for the real analogue of density of $S_k$ in $B$, none of these problems occur. Theorem \ref{theorem1} shows that the complex density criterion can be carried over to the reals \textit{without changing it}. Condition \ref{criterion} does not involve the real structures at all, applying to any real structure on the family. The result is density of $S_k \subset B$ \textit{and} $R_k \subset B({\bb R})$. It is for this reason that the applications of Theorem \ref{theorem1} are generous: the statements in Theorem \ref{theorem2}, as well as their proofs, are direct analogues of some applications of Theorem \ref{colpol} in \cite{Colombo1990}. 
\\
\\
Let us comment on the topologies appearing in Theorem \ref{theorem2}.\hyperlink{theoremA}{A}\&\hyperlink{theoremB}{B}. One may be inclined to believe that to obtain a \textit{real moduli space} (i.e. a reasonable topology on the set of real isomorphism classes), one equips the complex moduli space with an anti-holomorphic involution and considers the set of fixed points. This often fails, as it does for $\mr A_g^{\bb R}$ and $\mr M_g^{\bb R}$: there can be a complex algebraic variety carrying different real structures (e.g. $y^2 = x^3 + x$ and $y^2 = x^3-x$ in $\mr A_1^{\bb C}$), and there can be real points of the complex moduli space that do not represent any real variety (e.g. the example \cite[86]{silholsurfaces} in $\mr A_2^{\bb C}$ due to Shimura \cite{shimura1972}). The obstruction is that these moduli spaces are not fine - indeed, one solution is to cover $\mr A_g^{\bb C}$ and $\mr M_g^{\bb C}$ by spaces that rigidify the varieties in such a way that any real structure can be lifted to a real structure compatible with the rigidification: take the disjoint union of several fixed point sets of non-equivalent real structures on these covering spaces and quotient out by appropriate groups to obtain a semi-analytic structure on $\mr A_g^{\bb R}$ and on $\mr M_g^{\bb R}$. This was done by Gross and Harris \cite{grossharris} for $\mr A_g^{\bb R}$ using the Siegel space $\bb H_g \to \mr A_g^{\bb C}$ and by Sepp\"al\"a and Silhol \cite{seppalasilhol2} for $\mr M_g^{\bb R}$ using the Teichm\"uller space $\ca T_g \to \mr M_g^{\bb C}$. 
\\
\\
In the second part of this paper we put our density results in some perspective. The goal is to prove that the above topologies on $\mr A_g^{\RR}$ and $\mr M_g^{\RR}$ are natural in some sense. In complex geometry there are many ways to construct a moduli space of complex varieties, among which natural ones use algebraic stacks. We show that something similar holds over $\RR$, and study the real locus of an algebraic stack $\mr X$ of finite type over $\RR$. We define a topology on the set $|\mr X(\RR)|$ of isomorphism classes of $\mr X(\RR)$ in a way that generalizes the euclidean topology on $\mr X(\RR)$ when $\mr X$ is a scheme. If $\mr X$ admits a coarse moduli space $\mr X \to M$, $|\mr X(\RR)|$ should be thought of as the real analogue of the euclidean topology on $M(\CC)$. The point is that we cannot use $M(\RR)$ since this set will almost never be in bijection with $|\mr X(\RR)|$. For the stacks of abelian varieties $\ca A_g$ and curves $\ca M_g$, the bijections $|\ca A_g(\RR)| \cong \mr A_g^{\RR}$ and $|\ca M_g(\RR)| \cong \mr M_g^{\RR}$ are homeomorphisms, see Theorems \ref{th:homeomorphismmoduli} and \ref{th:homeomorphismmoduli2}.
\\
\\
Finally, we would like to remark that Colombo-Pirola's Theorem \ref{colpol} has been generalized by Ching-Li Chai \cite{Chai1998DensityOM}, who considers a variation of rational Hodge structures over a complex analytic variety and rephrases and answers the following question in the context of Shimura varieties: when do the points corresponding to members having extra Hodge cycles of a given type form a dense subset of the base? It should be interesting to investigate whether such a generalization can be carried over to the real numbers as well.

\subsection{Outline of the paper}
This paper is organized as follows. Section \ref{realfamily} is devoted to the proof of Theorem \ref{theorem1}. In Section \ref{satisfydensity} we satisfy Condition \ref{criterion} in the case of a universal local deformation of a polarized abelian variety. In Section \ref{realmoduli} we prove that any real structure on a complex manifold admitting a universal local deformation extends uniquely to a real structure on the local deformation. In Section \ref{realjacobians} we show that any real structure on a family of curves induces a real structure on the relative Jacobian. We prove Theorem \ref{theorem2} in Section \ref{provetheorem2}. In Section \ref{coarsemodulistsack} we define a topology on the real locus of a real algebraic stack. For algebraic moduli stacks over $\RR$, the induced topological space becomes a moduli space of real varieties. We conclude in Section \ref{realan} that for abelian varieties and curves, the so-obtained real moduli spaces coincide with those defined by Gross-Harris \cite{grossharris} and Sepp\"al\"a-Silhol \cite{seppalasilhol2}.

\subsection{Acknowledgements}
I would like to thank my thesis advisor Olivier Benoist for his great guidance, encouragement and support. I would also like to sincerely thank the referee for carefully reading the manuscript. His or her comments substantially improved the quality of this paper.

\section{Real Abelian Subvarieties in Family} \label{realfamily}

\subfile{realfamily}

\section{Density in Deformation Spaces} \label{satisfydensity}

\subfile{satisfydensity}

\section{Real Deformation Spaces} 
\label{realmoduli}

\subfile{realdeformationspaces}

\section{Real Structures on Relative Jacobians} \label{realjacobians}

\subfile{realjacobians}

\section{Proof of Theorem \ref{theorem2}} \label{provetheorem2}

\subfile{provetheorem2}

\section{The Coarse Moduli Space of a Real Algebraic Stack} \label{coarsemodulistsack}

\subfile{intermezzo}

\section{Comparing the Real Moduli Spaces} 
\label{realan}

\subfile{comparingmodulispaces}

\newgeometry{left=10mm, bottom=1in} 
\printbibliography
\end{document}

%% file: macros.tex
\makeatletter
\newcommand{\extp}{\@ifnextchar^\@extp{\@extp^{\,}}}
\def\@extp^#1{\mathop{\bigwedge\nolimits^{\!#1}}}
\makeatother

\newcommand\blfootnote[1]{%
  \begingroup
  \renewcommand\thefootnote{}\footnote{#1}%
  \addtocounter{footnote}{-1}%
  \endgroup
}

\newcommand*{\img}[1]{%
    \raisebox{-.0\baselineskip}{%
        \includegraphics[
        height=\baselineskip,
        width=\baselineskip,
        keepaspectratio,
        ]{#1}%
    }%
}

\newcommand{\white}{\;\;\;}

\newcommand{\Grass}{\textnormal{Grass}}

\newcommand{\Jac}{\text{Jac}}

\newcommand{\NL}{\textnormal{NL}}

\newcommand{\sm}{\setminus}

\newcommand{\CC}{\mathbb{C}}

\newcommand{\OO}{\mathcal{O}}

\newcommand{\QQ}{\mathbb{Q}}

\newcommand{\Sym}{\textnormal{Sym}}

\newcommand{\GL}{\textnormal{GL}}

\newcommand{\RR}{\mathbb{R}}

\newcommand{\ZZ}{\mathbb{Z}}

\newcommand{\Pic}{\textnormal{Pic}}

\newcommand{\Ker}{\textnormal{Ker}}

\newcommand{\id}{\textnormal{id}}

\newcommand{\Hom}{\textnormal{Hom}}

\newcommand{\pr}{\textnormal{pr}}
\newcommand{\Gal}{\textnormal{Gal}}

\newcommand{\Spec}{\textnormal{Spec }}

\newcommand{\Sp}{\textnormal{Sp}}

\newcommand{\ca}[1]{{\mathcal{#1}}}
\newcommand{\bb}[1]{{\mathbb{#1}}}

\newcommand{\mr}[1]{{\mathscr{#1}}}

\newcommand{\tn}[1]{{\textnormal{#1}}}

\let\mod\relax
\DeclareMathOperator{\mod}{\textnormal{mod}}

\DeclareSymbolFont{bbm}{U}{bbm}{m}{n}
\DeclareSymbolFontAlphabet{\mathbbm}{bbm}

%% file: realfamily.tex
In this section we shall provide the setup for our proof of Theorem \ref{theorem1}, which we present hereafter, in Section \ref{densityproofsection}. 
\\
\\
Let $\psi: \ca A \to B$ as in Section $1$. Let $\bb V_{{\bb Z}} =  R^1\psi_{*}{\bb Z}$ be the ${\bb Z}$-local system attached to $\psi$. The holomorphic vector bundle $\ca H  =  \bb V_{{\bb Z}}  \otimes_{{\bb Z}} \OO_{B}$ is endowed with a filtration by the holomorphic subbundle $F^1\ca H = \ca H^{1,0} \subset \ca H$, of fiber $(\ca H^{1,0})_t = H^{1,0}(A_t) \subset H^1(A_t, {\bb C})= \ca H_t$. Denote by $\ca H_{{\bb R}}$ the real $C^{\infty}$-subbundle of $\ca H$ whose fibers are $(\ca H_{{\bb R}})_t = H^1( A_t, {\bb R})$. Define $G = \Gal({\bb C}/{\bb R})$. Recall that $\ca A $ and $B$ are endowed with anti-holomorphic involutions $\tau$ and $\sigma$ such that $\psi \circ \tau = \sigma \circ \psi$. The map $\tau$ induces, for each $t \in B$, an anti-holomorphic isomorphism $\tau: A_t \cong A_{\sigma(t)}$. The pullback of $\tau$ gives an isomorphism of fibers $\tau^*: \ca H_{\sigma(t)} \to \ca H_{t}$ inducing an involution of differentiable manifolds
$$
F_{\infty}: \ca H \to \ca H 
$$
over the involution $\sigma: B \to B$. Composing $F_{\infty}$ fiberwise with complex conjugation provides an involution of differentiable bundles 
$$
F_{dR}: \ca H \to \ca H
$$ over $\sigma$ which respects the Hodge decomposition by \cite[I, Lemma 2.4]{silholsurfaces}. If we let $\mr G(k,\ca H^{1,0})$ be the complex Grassmannian bundle of complex $k$-planes in $\ca H^{1,0}$ over $B$, and $\mr G(2k, \ca H_{{\bb R}})$ the real Grassmannian bundle of real $2k$-planes in $\ca H_{{\bb R}}$ over $B$, we see that $G$ acts on these bundles via the morphisms $F_{dR}: \ca H^{1,0} \to \ca H^{1,0}$ and $\tau^\ast: \ca H_{{\bb R}} \to \ca H_{{\bb R}}$. Since the diffeomorphism $\ca H_{{\bb R}} \to \ca H^{1,0}$, defined as the composition of morphisms
$$
\ca H_{{\bb R}} \hookrightarrow \ca H = \ca H^{1,0} \oplus \ca H^{0,1} \to \ca H^{1,0},
$$
is $G$-equivariant, it induces a $G$-equivariant morphism of differentiable manifolds
\[
\mr G(k,\ca H^{1,0}) \to \mr G(2k, \ca H_{{\bb R}}).
\]
Let $0 \in B({\bb R})$ and choose a $G$-stable contractible neighbourhood $U$ of $0$ in $B$. Trivialize $\bb V_{\bb Z}$ over $U$, which trivializes $\mr G(2k, \ca H_{{\bb R}})$ over $U$, and consider the morphism $\Phi$ defined as the composite
$$\mr G(k,\ca H^{1,0})|_U \to \mr G(2k, \ca H_{{\bb R}})|_U \cong U \times \tn{Grass}_{{\bb R}}(2k, H^{1}(A_0, {\bb R})) \to \tn{Grass}_{{\bb R}}(2k, H^{1}(A_0, {\bb R})).$$ Let $j$ be the canonical map $\tn{Grass}_{{\bb Q}}(2k, H^{1}(A_0, {\bb Q})) \to \tn{Grass}_{{\bb R}}(2k, H^{1}(A_0, {\bb R}))$. We obtain the following diagram: 
\begin{equation} \label{Skdiagram}
\xymatrixcolsep{5pc}
\xymatrix{
\mr G(k,\ca H^{1,0})|_U\ar[d]^f \ar[r]^{\Phi \hspace{1cm}} & \tn{Grass}_{{\bb R}}(2k, H^{1}(A_0, {\bb R})) \\
U & \tn{Grass}_{{\bb Q}}(2k, H^{1}(A_0, {\bb Q})). \ar[u]_j
}
\end{equation}
\noindent
Write $U(\bb R) = U \cap B(\bb R)$. We shall show how diagram (\ref{Skdiagram}) provides the parametrization of polarized real abelian varieties containing a $k$-dimensional real abelian subvariety. 

\begin{proposition} \label{gequivprop00}
\begin{enumerate}  \label{gequivprop0}
\item  \label{gequivprop1}
The morphisms 
$
f$, $\Phi$ and $j$ in diagram (\ref{Skdiagram}) are $G$-equivariant. 
\item \label{gequivprop2} 
We have 
$$
f(\Phi^{-1}(j(\tn{Grass}_{{\bb Q}}(2k, H^{1}(A_0, \bb Q))^G))) = R_k \cap U({\bb R}),
$$
where $R_k = \{ t \in B({\bb R}): A_t \text{\textit{ contains an abelian subvariety over ${\bb R}$ of dimension }} k \}$. 
\end{enumerate}
\end{proposition}
\begin{proof} 
1. The fact that $f$ and $j$ are $G$-equivariant is immediate from the description of $F_{dR}$. For $\Phi$, it suffices to show that the trivialization $\ca H_{\bb R} \cong U \times H^1(A_0, \bb R)$ is $G$-equivariant, and this map is induced by the restriction $r: \bb V_{\bb R}|_U \to (\bb V_{\bb R})_0$ which is an isomorphism of local systems; but $r$ is unique if we require that $r$ induces the identity on $(\bb V_{\bb R})_0$. 
\\
\\
2. Let $t \in U({\bb R})$ and consider the polarized real abelian variety $(A_t, \tau)$. We have $A_t \cong V / \Lambda$, where $V \cong H^{1,0}(A_t)^{*}$ and $\Lambda \cong H_1(A_t, {\bb Z})$. It follows that $A_t$ contains a complex abelian subvariety $X$ of dimension $k$ if and only if there exists a $k$-dimensional ${\bb C}$-vector subspace $W_1 \subset H^{1,0}(A_t)$ and a $k$-dimensional $\bb Q$-vector subspace $W_2 \subset H^1(A_t, \bb Q)$ such that, under the canonical real isomorphism $H^{1,0}(A_t) \cong H^1(A_t, {\bb R})$, the space $W_1$ is identified with $W_2 \otimes {\bb R}$. In this case, $L = W_2^* \cap \Lambda$ is a lattice in $W_1^*$, and $X = W_1^{*} / L$. Then observe that the $k$-dimensional complex abelian subvariety $X \subset A_t$ is a $k$-dimensional real abelian subvariety $(X, \tau|_X)$ of $(A_t, \tau)$ if and only if $\tau(X)  = \tau(W_1^{*} / L) = W_1^{*} / L = X$. The latter is equivalent to $F_{dR}(W_2) = W_2$ by Lemma \ref{lemmatje} below, and we are done.
\end{proof}
\begin{lemma} \label{lemmatje}
Let $A$ be a complex torus, let $\Lambda = H_1(A, \ZZ)$ and consider the Hodge decomposition $\Lambda_\CC = V^{-1,0} \oplus V^{0,-1}$. For an anti-holomorphic involution $\sigma: A \to A$ such that $\sigma(e) = e$, let $F_\infty(\sigma): \Lambda \to \Lambda$ be the pushforward of $\sigma$, and let $F_{dR}(\sigma): V^{-1,0}  \to V^{-1,0} $ correspond to the differential $d\sigma: T_eA \to T_eA$. This defines a bijection between:
\begin{enumerate}
    \item[(i)] The set of real structures $\sigma: A \to A$.
    \item[(ii)] The set of involutions $F_\infty: \Lambda \to \Lambda$ such that $F_{\infty, \CC} \left( V^{-1,0}  \right) = V^{0, -1} $.
    \item[(iii)] The set of anti-linear involutions $F_{dR}: V^{-1,0}  \to V^{-1,0} $ such that $F_{dR} \left( \Lambda  \right) = \Lambda $.
\end{enumerate}
\end{lemma}

\begin{proof}
If $\sigma: A \to A$ is as in $(i)$, then $F_\infty(\sigma)_\CC$ interchanges the factors of the Hodge decomposition by \cite[I, (2.4)]{silholsurfaces}, so $\sigma \mapsto F_\infty(\sigma)$ is a well-defined map $(i) \to (ii)$. For $F_\infty: \Lambda \to \Lambda$ as in $(ii)$, the restriction to $V^{-1,0}$ of the composition of $F_{\infty, \CC} $ with complex conjugation defines a map $F_{dR} = \text{conj} \circ F_{\infty, \CC} : V^{-1,0} \to V^{-1,0}$ as in $(iii)$, which equals $F_{dR}(\sigma)$ when $F_\infty=F_\infty(\sigma)$. Then $F_\infty \mapsto \text{conj} \circ F_{\infty, \CC}$ determines $(ii)\cong (iii)$. Finally, the map $\sigma \mapsto F_{dR}(\sigma)$, $(i) \to (iii)$ is the map giving the natural correspondence between anti-holomorphic involutions on $A$ preserving $e$ and anti-holomorphic involutions on its universal cover $T_eA$ that preserve $0$ and are compatible with $\pi_1(A,e) = H_1(A, \ZZ)$-orbits.
\end{proof}

\subsection{Proving the density theorem} \label{densityproofsection}
\subfile{densityproofsection}

%% file: densityproofsection.tex
For a smooth manifold $W$ on which a compact Lie group $H$ acts by diffeomorphisms, the set of fixed points $W^H$ has a natural manifold structure that makes it a submanifold of $W$ \cite[Corollary I.2.3]{audin}. $\Phi$ in Diagram (\ref{Skdiagram}) is $G$-equivariant by Proposition \ref{gequivprop00}.\ref{gequivprop1}; denote
by $\Phi^G$ the induced morphism on fixed spaces. We obtain a diagram of $\ca C^{\infty}$-manifolds:
\begin{equation} \label{Skdiagram2}
\xymatrixcolsep{5pc}
\xymatrix{
\mr G(k,\ca H^{1,0})^G|_{U(\bb R)} \ar[d]^f \ar[r]^{\Phi^G\hspace{1cm}} & \tn{Grass}_{{\bb R}}(2k, H^{1}(A_0, {\bb R}))^G \\
U({\bb R}) &  \tn{Grass}_{{\bb Q}}(2k, H^{1}(A_0, {\bb Q}))^G. \ar[u]_j
}
\end{equation}
\noindent
Recall that we obtained the equality $$f(\Phi^{-1}(j(\tn{Grass}_{{\bb Q}}(2k, H^{1}(A_0, \bb Q))^G))) = R_k \cap U({\bb R})$$ in Proposition \ref{gequivprop00}.\ref{gequivprop2}. Recall also the symmetric bilinear form $$
q: H^{1,0}(A_t) \otimes H^{1,0}(A_t) \to T_t^*U
$$ from Section \ref{introduction}, given by the differential of the period map and the isomorphism $H^{0,1}(A_t) \cong H^{1,0}(A_t)^*$ which the polarization induces. Finally, recall the notation (Equation (\ref{rk}), \S \ref{introduction}) $$ R_k  = \{ t \in B({\bb R}): A_t \text{\textit{ contains an abelian subvariety over ${\bb R}$ of dimension }} k \}.$$
Let us fix notation and introduce the aim of this section. 
\begin{notation}
For $t \in B$ and $W \in \Grass_{\bb C}(k, H^{1,0}(A_t))$, let $W^\perp$ denote the orthogonal complement of $W$ in $H^{1,0}(A_t)$ with respect to the polarization $E_t$. Define the sets $\mr E_k, \mr F_k \subset \mr G(k, \ca H^{1,0})$ and $\mr E_k(\bb R), \mr F_k(\bb R),\mr R_{k,U} \subset \mr G(k, \ca H^{1,0})^G$ as follows:
\begin{equation}
\mr E_k = \{(t, W) \in \mr G(k, \ca H^{1,0}) \tn{ : } 0 \to W \otimes W^\perp \to T_t^*B \ \tn{ is exact} \}
\end{equation}
\begin{equation}
\mr F_k = \{(t, W) \in \mr G(k, \ca H^{1,0}) \tn{ : } 0 \to  \wedge^2 W\to W \otimes H^{1,0}(A_t) \to T_t^*B  \tn{ is exact} \} 
\end{equation}
\begin{equation}
\mr E_k(\bb R) = \mr E_k \cap \mr G(k, \ca H^{1,0})^G
\end{equation}
\begin{equation}
\mr F_k(\bb R) = \mr F_k \cap \mr G(k, \ca H^{1,0})^G
\end{equation}
\begin{equation}
\mr R_{k,U} = \Phi^{-1}(j (\tn{Grass}_{\bb Q}(2k, H^{1}(A_0, \bb Q))^G)).
\end{equation}
\end{notation}
\noindent
Note that $f(\mr R_{k,U}) = R_k \cap U(\bb R)$. Our strategy to prove Theorem \ref{theorem1} is the following:
\begin{proposition} \label{propper}
Let $\mr E_k, \mr F_k , \mr R_{k,U} \subset \mr G(k, \ca H^{1,0})$ be as above. One has inclusions
$$
\xymatrix{
\mr F_k(\bb R)|_{U(\bb R)} \white \ar@{^{(}->}[r] \white &\white \mr E_k(\bb R)|_{U(\bb R)} \white \ar@{^{(}->}[r]\white&\white \overline{\mr R_{k,U}} \white\ar@{^{(}->}[r]\white& \white \mr G(k, \ca H^{1,0})^G|_{U(\bb R)} .
}
$$
If $\mr F_k$ is not empty then $\mr F_k(\bb R)$ is dense in $\mr G(k, \ca H^{1,0})^G$. Consequently, if $\mr F_k \neq \emptyset$, then $\mr R_{k,U}$ is dense in $\mr G(k, \ca H^{1,0})^G|_{U(\bb R)}$; in particular, then $R_k \cap U(\bb R)$ is dense in $U(\bb R)$.
\end{proposition}
\noindent
Before proving Proposition \ref{propper} we remark that it implies Theorem \ref{theorem1}. 

\begin{proof}[Proof of Theorem \ref{theorem1}]
It suffices to show that for each $x \in B(\bb R)$ and any $G$-stable contractible open neighborhood $x \in U \subset B$, the set $R_k \cap U \cap B(\bb R)$ is dense in $U \cap B(\bb R)$. Thus let $x \in U \subset B$ be such a real point and complex open neighborhood. Condition \ref{criterion} is satisfied if and only if $\mr F_k$ is non-empty; by Proposition \ref{propper}, we are done.
\end{proof}

\begin{proof}[Proof of Proposition \ref{propper}]
Let us first prove the inclusions. We have $\mr F_k \subset \mr E_k$: if $(t, W) \in \mr F_k$ then $\Ker( W \otimes H^{1,0}(A_t) \to T_t^*B) \subset W \otimes W$. Hence $\mr F_k(\bb R)|_{U(\bb R)} \subset \mr E_k(\bb R)|_{U(\bb R)} $. \\
\\
For the second inclusion, consider again the map $\Phi: \mr G(k,\ca H^{1,0})|_U \to \Grass_{\bb R}(2k, H^1(A_0, \bb R))$ as in Diagram (\ref{Skdiagram}). Define a set $Y \subset \mr G(k,\ca H^{1,0})$ as follows:
$$
Y = \{ x \in \mr G(k,\ca H^{1,0})|_U \tn{ \textit{such that the rank of} } d\Phi \tn{ \textit{is maximal at} }x  \}  \white \subset \white \mr G(k, \ca H^{1,0})|_U.
$$
Then $ Y = \mr E_k|_U$ by \cite[\S 1]{Colombo1990}. Let $Z$ be the set of points $x \in \mr G(k,\ca H^{1,0})^G|_{U(\bb R)}$ such that the rank of $d(\Phi^G): T_x\mr G(k,\ca H^{1,0})^G|_{U(\bb R)}  \to T_{\Phi(x)}\tn{Grass}_{{\bb R}}(2k, H^{1}(A_0, {\bb R}))^G$ is maximal at $x$. We claim that $
Z = Y \cap \mr G(k,\ca H^{1,0})^G|_{U(\bb R)}
$. Indeed, for a point $x \in \mr G(k, \ca H^{1,0})^G|_{U(\bb R)}$, the rank of $d(\Phi^G)$ is maximal at $x$ if and only if the rank of $d \Phi$ is maximal at $x$, since $$T_x\mr G(k, \ca H^{1,0})  = (T_x\mr G(k, \ca H^{1,0}))^G \otimes_{{\bb R}} {\bb C} = T_x\mr G(k, \ca H^{1,0})^G \otimes_{{\bb R}} {\bb C},$$ and similarly for $T_{\Phi(x)} \tn{Grass}_{{\bb R}}(2k,H^1(A_0, {\bb R}))$. Moreover, the differential $d\Phi$ at the fixed point $x \in \mr G(k,\ca H^{1,0})^G|_{U(\bb R)}$ is simply the complexification of the differential $d( \Phi^G)$ at $x$.
\\
\\
Consequently, $Z = Y \cap \mr G(k,\ca H^{1,0})^G|_{U(\bb R)} = \mr E_k \cap \mr G(k, \ca H^{1,0})^G|_{U(\bb R)} = \mr E_k(\bb R)|_{U(\bb R)}$. It follows that the morphism
$$
\Phi^G: \mr G(k,\ca H^{1,0})^G|_{U(\bb R)} \to \tn{Grass}_{{\bb R}}(2k, H^{1}(A_0, {\bb R}))^G
$$
is open when restricted to the set $\mr E_k(\bb R)|_{U(\bb R)}$. Since $j(\tn{Grass}_{{\bb Q}}(2k, H^{1}(A_0, {\bb Q}))^G)$ is dense in $\tn{Grass}_{{\bb R}}(2k, H^{1}(A_0, {\bb R}))^G$ by Lemma \ref{lemmabetween} below, $\Phi^{-1}( j(\tn{Grass}_{{\bb Q}}(2k, H^{1}(A_0, {\bb Q}))^G)) \cap \mr E_k(\bb R)|_{U(\bb R)}$ is dense in $\mr E_k(\bb R)|_{U(\bb R)}$. But $\Phi^{-1}( j(\tn{Grass}_{{\bb Q}}(2k, H^{1}(A_0, {\bb Q}))^G)) \cap \mr E_k(\bb R)|_{U(\bb R)} = \mr R_{k,U} \cap  \mr E_k(\bb R)|_{U(\bb R)}$, so $\mr R_{k,U} \cap  \mr E_k(\bb R)|_{U(\bb R)}$ is dense in $ \mr E_k(\bb R)|_{U(\bb R)}$. In particular, $ \mr E_k(\bb R)|_{U(\bb R)}$ is contained in the closure of $\mr R_{k,U}$ in $ \mr G(k,\ca H^{1,0})^G|_{U(\bb R)} $. The inclusions are thus proved.
\\
\\
It remains to prove the assertion that non-emptiness of $\mr F_k$ implies density of $\mr F_k(\bb R)$ in $\mr G(k, \ca H^{1,0})^G$. Write $\mr G = \mr G(k, \ca H^{1,0})$ and $\mr G(\RR) = \mr G^G$. Let $Z_k \subset \mr G $ be the complement of $\mr F_k$ in $\mr G$. Observe that $Z_k$ equals the set of pairs $(t,W)\in \mr G $ such that the injection $\bigwedge^2 W \to \Ker ( W  \otimes H^{1,0}(A_t) \to T_t^*B )$ is not an isomorphism. This latter map varies holomorphically; therefore, locally on $\mr G$, the set of points where the rank is not maximal is determined by the vanishing of minors in a matrix with holomorphic coefficients. Since $\mr G$ is connected it follows that $Z_k$ is nowhere dense in $\mr G$. If $n = \dim_\CC \mr G$ then $\mr G(\RR)$ is a closed real submanifold of $\mr G$ of real dimension $n$ and we claim that $Z_k(\RR)$ is nowhere dense in $\mr G(\RR)$. Suppose for contradiction that $Z_k(\RR)$ contains a non-empty open set $V \subset \mr G(\RR)$. Locally around $t\in V$, $t \in V \subset \mr G$ is the inclusion of the real locus $0 \in B(\RR)$ of a euclidean open ball $0 \in B \subset \CC^n$ such that $B \cap Z_k = \{f_1 =  \dotsc = f_m = 0 \}$ for some holomorphic functions $f_i: B \to \CC$. Since the $f_i$ vanish on $B(\RR)$, they vanish on $B$ hence $B \subset Z_k$. 
\end{proof}

\begin{lemma}\label{lemmabetween} Let $n, k \in \bb Z_{\geq 0}$. Let $V$ be an $n$-dimensional vector space over $\bb Q$ and consider a linear transformation $F \in \textnormal{GL}(V)_{\bb Q}$ which is diagonalizable over $\bb Q$. For $L \in \{\QQ, \RR\}$, denote by $\bb G(k, V)^F(L)$ the set of $k$-dimensional $L$-subvector spaces $W \subset V \otimes_{\bb Q} L$ for which $F(W) = W$. Then $\bb G(k, V)^F(\bb Q)$ is dense in $\bb G(k, V)^F(\bb R)$. 
\end{lemma}

\begin{proof} 
For $F = \id$ this is an elementary fact. In order to deduce the general case from this, let $\lambda_1, \dotsc, \lambda_r \in \bb Q^*$ be the eigenvalues of $F$, and denote by $V_i : = V^{F = \lambda_i} \subset V, i \in I : =  \{1, \dotsc, r \}$ the corresponding eigenspaces. Eigenspaces are preserved under scalar extension so $(V \otimes_{\bb Q} \bb R)^{F_{\bb R} = \lambda_i} = V_i \otimes_{\bb Q} \bb R$. But a $k$-dimensional $\bb R$-subvector space $W \subset V_{\bb R}$ satisfies $F(W) = W$ if and only if $W = \bigoplus_{i \in I}W_i$ with $W_i$ a $\bb R$-subvector space of $V_i\otimes_{\bb Q} \bb R$ for each $i \in I$ and $\sum_{i \in I} \dim W_i = k$, and $W$ is defined over $\bb Q$ if and only if every $W_i$ is. This means that under the canonical diffeomorphism 
\begin{equation}
\bb G(k, V)^F(\bb R)  \cong \underset{\sum k_i = k}{\bigsqcup} \prod_{i \in I} \bb G(k_i, V_i)(\bb R) 
\end{equation}
the rational subspace $ \bb G(k, V)^F(\bb Q) $ is identified with $\bigsqcup_{k_i} \prod_{i \in I} \bb G(k_i, V_i)(\bb Q)$. 
\end{proof}

%% file: satisfydensity.tex
Let  
$
\psi$ be a family such as family (\ref{family}) in \S \ref{introduction}: we assume $\psi$ to be a holomorphic family $\psi: \ca A \to B$ of complex abelian varieties of dimension $g$, polarized by a section $E \in R^2\psi_*\bb Z$. For $t \in B$, denote by $\omega_t \in H^1(A_t, \Omega^1_{A_t})$ the K\"ahler class corresponding to the polarization $E_t \in H^2(A_t, \bb Z)$. Let $t \in B$, define $X = \psi^{-1}(t)$, and suppose that $\psi$ is a universal local deformation of $(X, \omega_t)$: the Kodaira-Spencer map gives $\rho: T_tB \cong H^1(X, T_X)_{\omega}$, where $H^1(X, T_X)_{\omega}$ is defined to be the kernel of $H^1(X, T_X) \to H^2(X, T_X \otimes \Omega^1_X) \to H^2(X, \OO_X)$, the first map being the cup-product with $\omega_t$ and the second induced by $T_X \otimes \Omega^1_X \to \OO_X$. 

\begin{proposition} \label{satisfycondition}
Condition \ref{criterion} is satisfied for $t \in B$ and any $W \in \textnormal{Grass}_\CC(k, H^{1,0}(A_t))$. 
\end{proposition}
\begin{proof}

By a theorem of Griffiths \cite[Th\'eor\`eme 17.7]{voisin}, the dual $q^*$ of the bilinear form $$q: H^{1}(X, \OO_X)^* \otimes H^{1}(X, \OO_X)^* = H^{1,0}(X) \otimes H^{1,0}(X)  \to T_t^*B$$ (see Equation (\ref{symbil}) in \S \ref{introduction}) factors as
$$
T_tB \to H^1(X, T_X) \to \Hom(H^{0}(X, \Omega^1_X) , H^{1}(X, \OO_X))\xrightarrow{\sim} H^{1}(X, \OO_X) \otimes H^{1}(X, \OO_X). 
$$
The second arrow is an isomorphism. The third arrow is induced by the polarization. We then remark that the following diagram commutes, and that the first row is exact:
\begin{equation}\label{com}
\xymatrixcolsep{3pc}
\xymatrix{
0 \ar[r]  & H^1(X, T_X)_{\omega} \ar[r]  &  H^1(X, T_X)  \ar[r] \ar[d]^{\rotatebox{90}{$\sim$}} & H^2(X, \OO_X) \ar[r]  & 0 \\
0 \ar[r]& T_tB \ar[u]_{\rho}\ar[r]^{q^* \hspace{1cm}} & H^{1}(X, \OO_X) \otimes H^{1}(X, \OO_X) \ar[r] & \bigwedge^2H^{1}(X, \OO_X) \ar[r] \ar[u]_{\rotatebox{90}{$\sim$}}  & 0.
}
\end{equation}
We conclude that the Kodaira-Spencer map $\rho$ is an isomorphism if and only if the second row in (\ref{com}) is exact if and only $q^*$ induces an isomorphism 
\begin{equation} \label{analytickodaira}
q^*: T_tB \to \Sym^2H^1(X, \OO_X).
\end{equation}
Identify $T_t^*B$ with $\Sym^2H^{1,0}(X)$ and $q$ with $H^{1,0}(X) \otimes H^{1,0}(X) \to \Sym^2H^{1,0}(X)$, and observe that for any complex vector space $V$ and \textit{every} $k$-dimensional subspace $W \subset V$, the natural sequence $0 \to\bigwedge^2W \to W \otimes V \to \text{Sym}^2(V)$ is exact.
 \end{proof}

%% file: realdeformationspaces.tex
In this section we prove that for a universal local deformation of a complex manifold $X$, any real structure on $X$ extends uniquely to a real structure on the local deformation. 
\\
\\
So let $X$ be a compact complex manifold, possibly polarized by the first Chern class $\omega = c_1(\ca L) \in H^2(X, \ZZ)$ of an ample line bundle $\ca L$. Consider a universal local deformation $\pi: \mr X \to B \ni 0$ of $X$ (resp. of $(X, \omega)$), where $B$ is any complex analytic space. Let $\sigma: X \to X$ be an anti-holomorphic involution, compatible with $\omega$ in case $X$ is polarized. Only the germ of $\pi$ in $0 \in B$ plays a role and all statements should be read in this sense.

\begin{proposition}[compare \cite{catanese_frediani_2003}, Section 4] \label{realdef}The real structure $\sigma : X \to X$ extends uniquely to a real structure on the universal local deformation $\pi: \mr X \to B$. In other words, possibly after restricting $(B,0)$ there is a unique couple of anti-holomorphic involutions $\tau: B \to B$, $\mr T: \mr X \to \mr X$ such that $\pi \circ \mr T = \tau \circ \pi$, $\tau(0) = 0$ and $\mr T|_X = \sigma: X \to X$.  
\end{proposition}

\begin{proof}
Consider the complex conjugate analytic spaces $X^\sigma$, $B^\sigma$ and $\mr X^\sigma$ of $X, B$ and $\mr X$ respectively (see \cite[Ch.I, Definition 1.1]{silholsurfaces} for the definition of complex conjugate analytic variety; the definition for general analytic spaces is similar). There is an induced local deformation $\pi^\sigma: \mr X^\sigma \to B^\sigma\ni 0$ of the manifold $\left(\mr X^\sigma \right)_0 = \left(\mr X_0 \right)^\sigma = X^\sigma$. It is easily shown that $\pi^\sigma$ is a universal local deformation of $X^\sigma$. Now the anti-holomorphic involution $\sigma: X \to X$ induces a biholomorphic function $\phi: X^\sigma \to X$ such that $ \phi \circ \phi^\sigma = \id: X \to X$, hence the fibers of $\pi$ and $\pi^\sigma$ above $0$ are isomorphic via $\phi$. By the universal properties of $\pi$ and $\pi^\sigma$, this means that possibly after restricting $B$ around $0$, there is a unique pair of biholomorphisms $f: B^\sigma \to B$, $g: \mr X^\sigma \to \mr X$ making the following diagram cartesian: 
\begin{equation}
\xymatrix{
\mr X^\sigma \ar[r]^{g} \ar[d]^{\pi^\sigma} & \mr X \ar[d]^\pi \\
B^\sigma \ar[r]^{f} & B.
}
\end{equation}
Moreover, $f(0) = 0$ and $g|_{X^\sigma}: X^\sigma \to X$ is the map $\phi$. Applying the functor $(-)^\sigma$ to this diagram, we obtain a pair of cartesian diagrams
\begin{equation} \label{getadiagram}
\xymatrix{
\mr X \ar[d]^\pi\ar[r]^{g^\sigma } & \mr X^\sigma \ar[r]^{g} \ar[d]^{\pi^\sigma} & \mr X \ar[d]^\pi \\
B \ar[r]^{f^\sigma } & B^\sigma \ar[r]^{f} & B
}
\end{equation}
such that $(f \circ f^\sigma)(0) = 0$ and such that $(g \circ g^\sigma)|_{X}: X \to X^\sigma \to X$ is the map $\phi \circ \phi^\sigma = \id$. But this means that $f \circ f^\sigma = \id$ and $g \circ g^\sigma = \id$. Composing $f^\sigma: B \to B^\sigma$ (resp. $g^\sigma: \mr X \to \mr X^\sigma$) with the canonical anti-holomorphic map $B^\sigma \to B$ (resp. $\mr X^\sigma \to \mr X$) gives the desired anti-holomorphic involution $\tau: B \to B$ (resp. $\mr T: \mr X \to \mr X$). 
\end{proof}

%% file: realjacobians.tex
The goal of this section is to prove that a real structure on a family of curves extends canonically to a real structure on the corresponding family of Jacobians.
\\
\\
Let $B$ be a simply connected complex manifold and let $\pi: \mr X \to B$ be a family of compact genus $g$ Riemann surfaces. The relative Jacobian
\begin{equation} \label{jacobian} 
\psi: J_{\mr X} =  \underline{\Pic^0}(\mr X/B) \to B
\end{equation}
of $\pi: \mr X \to B$ can be constructed by considering the exact sequence of sheaves on $\mr X$:
\begin{equation} \label{gsheaves}
0 \longrightarrow \bb Z \longrightarrow \OO_{\mr X} \xrightarrow{\text{exp}2\pi i} \OO_{\mr X}^* \longrightarrow 0.
\end{equation}
Indeed, sequence (\ref{gsheaves}) induces an inclusion $R^1 \pi _* \bb Z \to R^1 \pi _* \OO_{\mr X}$ where we identify $R^1 \pi _* \bb Z$ with its \'etal\'e space and $R^1 \pi _* \OO_{\mr X}$ with its corresponding holomorphic vector bundle; define $J_{\mr X} = R^1 \pi _* \OO_{\mr X}/ R^1 \pi _* \bb Z$. We claim that family (\ref{jacobian}) is a polarized holomorphic family of $g$-dimensional complex abelian varieties in the sense of family (\ref{family}) in Section \ref{introduction}. Indeed, $\psi: J_{\mr X} \to B$ admits a section $s: B \to J_{\mr X}$: since $J_{\mr X}$ represents the relative Picard functor of degree $0$, it corresponds to $\OO_{J_{\mr X}} \in \Pic^0({\mr X})$. For $0 \in B$, we have the Riemann form $E_{0} = - \langle, \rangle: \bigwedge^2H_1({\mr X}_{0}, \bb Z) \to \bb Z$. The trivialization $R^2\pi_*\bb Z \cong H^2(J_{{\mr X}_{0}}, \bb Z)$ implies that 
$$
\Hom_{\bb Z}(\wedge^2  H_1({\mr X}_{0}, \bb Z), \bb Z) = \Hom_{\bb Z}( \wedge^2  H_1(\Jac({\mr X}_{0}), \bb Z), \bb Z)  =H^2({J_{\mr X}}_{0}, \bb Z)  \cong  H^0(B, R^2\pi_*\bb Z).
$$
In this way, $E_{0}$ extends to a principal polarization on the Jacobian family (\ref{jacobian}).

\begin{lemma} \label{realjacobianstructure}
Consider the Jacobian family $(\psi: J_{\mr X} \to B, s: B \to J_{\mr X}, E \in R^2\psi_*\bb Z)$ defined above. Each real structure $(\sigma: B \to B, \sigma': \mr X \to \mr X)$ on the curve $\pi: \mr X \to B$ induces a real structure on the relative Jacobian $\psi: J_{\mr X} \to B$: there exists an anti-holomorphic involution $\Sigma: J_{\mr X} \to J_{\mr X}$ such that
 $$(i) \white \psi \circ \Sigma = \sigma \circ \psi, \white\white (ii) \white \Sigma \circ s = s \circ \sigma, \white\white (iii)\white \Sigma^\ast(E) = -E.$$
\end{lemma}
\begin{proof}
Write $\bb V_{{\bb Z}} = R^1 \pi _* {\bb Z}$ and $\ca H = \bb V_{{\bb Z}}  \otimes_{{\bb Z}} \OO_{B}$. Then $\ca H$ is endowed with a filtration by the holomorphic subbundle $F^1\ca H  \subset \ca H$, and we have $\ca H / F^1\ca H = \ca H^{0,1} = R^1 \pi _* \OO_{\mr X}$. By the strategy in Section \ref{realfamily}, the real structure $(\sigma, \sigma')$ on $\pi: \mr X \to B$ induces an anti-holomorphic involution 
$
F_{dR}: \ca H \to \ca H$ compatible with $\sigma$ and preserving $\ca H^{0,1}$ and $\bb V_{{\bb Z}}$. Since $J_{\mr X} =  \ca H^{0,1} / \bb V_{{\bb Z}} $, the involution $F_{dR}$ induces an anti-holomorphic involution 
$
\Sigma: J_{\mr X} \to J_{\mr X}.  
$
By construction of $F_{dR}$ in Section \ref{realfamily}, we have $\psi \circ \Sigma = \sigma \circ \psi$. For $t \in B$, $\Sigma$ induces an anti-holomorphic map $\Sigma: \Jac(\mr X_{\sigma(t)}) \to \Jac(\mr X_{t})$. We need to prove that 
$
\Sigma^\ast(E_t) = - E_{\sigma(t)}$ for $ \Sigma^\ast: H^2( \Jac({\mr X}_t) , \bb Z) \to H^2( \Jac(\mr X_{\sigma(t)}) , \bb Z)$. This follows from the commutativity of 
$$
\xymatrixcolsep{5pc}
\xymatrix{
H^1({\mr X}_t, {\bb R}) \otimes H^1({\mr X}_t, {\bb R}) \ar[d]^{\Sigma^\ast \otimes \Sigma^\ast} \ar[r]^{\hspace{1cm}\cup} & H^2({\mr X}_t, {\bb R}) \ar[d]^{\Sigma^\ast} \ar[r]^{\sim} & {\bb R} \ar[d]^{-1} \\
  H^1(\mr X_{\sigma(t)}, {\bb R}) \otimes   H^1(\mr X_{\sigma(t)}, {\bb R})  \ar[r]^{\hspace{1cm}\cup} &   H^2(\mr X_{\sigma(t)}, {\bb R}) \ar[r]^{\sim} &  {\bb R}.  } 
$$
The left hand diagram commutes since pullback commutes with cup-product, and the right hand diagram commutes because $\Sigma$ is anti-holomorphic thus reverses the orientation. 
\end{proof}

%% file: provetheorem2.tex
The goal of this section is indeed to prove Theorem \ref{theorem2} using the results of the previous sections. The structure of Section \ref{provetheorem2} is as follows. We start by recalling the construction of the topology on $\mr A_g^\RR$ as in \cite{grossharris} and \cite{silholsurfaces}, and once that is done, we prove Theorem \ref{theorem2}.\textcolor{blue}{A}. We proceed by recalling the construction of the topology on $\mr M_g^\RR$ as in \cite{seppalasilhol2} and then prove Theorem \ref{theorem2}.\textcolor{blue}{B}. Finally, we finish Section \ref{provetheorem2} by proving Theorem \ref{theorem2}\textcolor{blue}{.C}.

\subsection{Density in $\mr A_g^\RR$} \label{sec:densityinag}

The set $\mr A_g^\RR$ of isomorphism classes of principally polarized real abelian varieties of dimension $g$ can be provided with a real semi-analytic space structure as follows. The result seems to have been proven independently by to Gross-Harris \cite[Section 9]{grossharris} and Silhol \cite[Ch.IV, Section 4]{silholsurfaces}. Following \cite[Ch.IV, Definition 4.4]{silholsurfaces}, define $I \subset \bb Z^2$ as 
$$I : = \left\{(\alpha, \lambda) \in \bb Z^{2}:  1 \leq \lambda \leq g  \tn{ and }  \alpha \in \{1,2\} \tn{ such that } \alpha = 1 \tn{ if } \lambda \tn{ is odd } \right\} \cup \left\{ (0,0) \right\}.$$ Attach to each $i = (\alpha,\lambda) \in I$ a matrix $M \in M_{g}(\bb Z)$ as in \cite[Ch.IV, Theorem 4.1]{silholsurfaces}, define $
\Gamma_i  =  \{ A \in \GL_g(\bb Z): A M A^t = M \} \subset \GL_g(\bb Z) \subset \Sp_{2g}(\bb Z)$, where the inclusion 
$\GL_g(\bb Z) \hookrightarrow \Sp_{2g}(\bb Z)$ is defined by $  A \mapsto \big(\begin{smallmatrix}
A &  0 \\
0  & A^{-t}
\end{smallmatrix}\big)$, and define an anti-holomorphic involution $\tau_i : \bb H_g \to \bb H_g$ on the Siegel space $\bb H_g$ by $\tau_i(Z) = M - \overline{Z}$. Let $\bb H_g^{\tau_i}$ be its fixed locus. Then the period map defines a bijection \cite[Proposition 9.3]{grossharris}, \cite[Ch.IV, Theorem 4.6]{silholsurfaces}
\begin{equation} \label{eq:topologyaggrossharris}
\mr A_g^{\bb R} \cong \bigsqcup_{i \in I} \Gamma_i \backslash \bb H_g^{\tau_i}.
\end{equation}

\begin{proof}[Proof of Theorem \ref{theorem2}.\textcolor{blue}{A}]
Let $i \in I$ and $Z \in \bb H_g^{\tau_i}$. Let $(X, \omega\in H^2(X, \ZZ))$ be the principally polarized complex abelian variety with symplectic basis whose period matrix is $Z$. Then $(X, \omega)$ admits a unique real structure $\sigma: X \to X$, compatible with $\omega$ and the symplectic basis \cite[Section 9]{grossharris}. There exists a $\tau_i$-invariant connected open neighborhood $B \subset \bb H_g$ of $Z$ and a universal local deformation $\pi: \mr X \to B \ni Z$ of the polarized complex abelian variety $(X, \omega)$. By Proposition \ref{realdef}, possibly after restricting $B$ around $Z$, the real structure $\sigma: X \to X$ extends uniquely to a real structure on the polarized family $\pi$, which by uniqueness is compatible with $\tau_i : B \to B$. By Proposition \ref{satisfycondition}, Condition \ref{criterion} is satisfied. By Theorem \ref{theorem1}, $R_k \cap B^{\tau_i} $ is dense in $B^{\tau_i} \subset \bb H_g^{\tau_i}$. It follows that $R_k$ is dense in $\bb H_g^{\tau_i}$. 
\end{proof}

\subsection{Density in $\mr M_g^\RR$}
Similarly, Sepp\"al\"a and Silhol provide the set $\mr M_g^\RR$ of real genus $g$ algebraic curves with a topology as follows. Fix a compact oriented $\ca C^{\infty}$-surface $\Sigma$ of genus $g$ and let $\ca T_g$ be the Teichm\"uller space of the surface $\Sigma$ (see e.g. \cite{arbarelloteichmuller}). Define $J \subset \bb Z^2$ to be the set of tuples $(\epsilon, k) \in \bb Z^2$ with $\epsilon \in \{0,1\}$ such that $1 \leq \lambda \leq g+1$ and $k \equiv g + 1 \mod 2$ when $\epsilon = 1$, and $0 \leq k \leq g$ otherwise. To every $j = (\epsilon(j), k(j)) \in J$ one can attach an orientation-reversing involution $\sigma_j: \Sigma \to \Sigma$ of \textit{type $j \in J$} \cite{seppalasilhol2}. This means that $k(j) = \#\pi_0(\Sigma^{\sigma_j})$ and that $\epsilon(j) = 0$ if and only if $\Sigma \setminus \Sigma^{\sigma_j}$ is connected. Moreover, every such involution $\sigma_j: \Sigma \to \Sigma$ induces an anti-holomorphic involution $\sigma_j: \ca T_g \to \ca T_g$ [\textit{loc. cit.}]. Denote by $N_j = \{ g \in \Gamma_g: g \circ \sigma_j = \sigma_j \circ g \}$ the normalizer of $\sigma_j: \ca T_g \to \ca T_g$ in the mapping class group $\Gamma_g$ of $\Sigma$. Then there is a natural bijection \cite[Theorem 2.1, Definition 2.3]{seppalasilhol2} 
\begin{equation} \label{eq:topologymgseppalasilhol}
    \mr M_g^{\bb R} \cong \bigsqcup_{j \in J} N_j \backslash \ca T_g^{\sigma_j}.
\end{equation}

\begin{proof}[Proof of Theorem \ref{theorem2}.\textcolor{blue}{B}] 

Suppose that $g \geq 3$, let $j \in J$ and consider a point $0 \in \ca T_t^{\sigma_j}$. Let $(X, [f])$ be the complex Teichm\"uller curve of genus $g$ corresponding to the point $0$. By \cite{seppalasilhol2}, there is a unique real structure $\sigma: X\to X$ which is compatible with the Teichm\"uller structure $[f]$ and the involution $\sigma_j: \Sigma \to \Sigma$. Moreover, there exists a $\sigma_j$-invariant simply connected open subset $B \subset \ca T_g$ of $0$ in the Teichm\"uller space and a Kuranishi family $\pi: \mr X \to B \ni 0$ of the Riemann surface $X$. By Proposition \ref{realdef}, up to restricting $B$ around $0$, the real structure $\sigma: X\to X$ extends uniquely to a real structure $(\tau, \mr T)$ on the Kuranishi family $\pi$ such that $\tau(0) = 0$. By uniqueness, $\tau: B \to B$ coincides with $\sigma_j$. By Lemma \ref{realjacobianstructure}, $(\tau, \mr T)$ induces a real structure $(\tau, \Sigma)$ on the Jacobian $J_{\mr X} \to B$ of the curve $\pi: \mr X \to B$. Fix $k \in \{1,2,3\}$. Observe that, by Theorem \ref{theorem1}, it suffices to prove that Condition \ref{criterion} holds in $B$. That is, we need to show that there exists an element $t \in B$ and a $k$-dimensional complex subspace $W \subset H^{1,0}(\Jac(\mr X_t))  = H^{1,0}(\mr X_t)$ such that the sequence 
$
0 \to \bigwedge^2 W \to W \otimes H^{1,0}(\mr X_t) \to T_t^*B$ is exact. The family
$
\pi: \mr X \to B
$
is a universal local deformation of $\mr X_t$ for each $t \in B$, hence $ T_tB \cong H^1(\mr X_t, T_{\mr X_t})$. By \cite[Lemme 10.22]{voisin}, the dual of $q: H^{1,0}(\mr X_t) \otimes H^{1,0}(\mr X_t) \to T_t^\ast{B}$ is nothing but the cup-product $
H^0(K_{\mr X_t}) \otimes H^0(K_{\mr X_t}) \to H^0(K_{\mr X_t}^{\otimes 2})$. Consequently, we are reduced to the claim that for each $k \in \{1,2,3\}$ there exists an element $t \in B$ and a $k$-dimensional subspace $W \subset H^0(K_{\mr X_t})$ such that the following sequence is exact:
\begin{equation} \label{exseq}
0 \to \wedge^2 W \to W \otimes H^0(K_{\mr X_t}) \to H^0(K_{\mr X_t}^{\otimes 2}).
\end{equation}
This is true by the \textit{Proof of Theorem (3)} in \cite{Colombo1990}. Indeed, Colombo and Pirola consider the moduli space of complex genus $g \geq 3$ curves $\mr M_g^\CC = \Gamma_g \sm \ca T_g$ to prove the complex analogue of Theorem \ref{theorem2}.\hyperlink{theoremB}{B}. They show that there exists a point $p = [C] \in \mr M_g^\CC$ and a $k$-dimensional complex subspace $W \subset H^0(K_C)$ such that (\ref{exseq}) is exact. So Condition \ref{criterion} is satisfied for some point $t \in \ca T_g$. Since Condition \ref{criterion} is open for the Zariski topology on $\ca T_g$, it is dense for the euclidean topology, hence Condition \ref{criterion} holds for some $t \in B$.  
\end{proof}

\subsection{Density of real plane curves covering an elliptic curve}

\begin{proof}[Proof of Theorem \ref{theorem2}\textcolor{blue}{.C}] 
We need to prove that the subset $\mr R_k(V)$ of $V$ of real plane curves that map non-trivially to elliptic curves over $\bb R$ is dense in $V$, where $V \subset \bb PH^0(\bb P^2_{{\bb R}}, \OO_{\bb P^2_{{\bb R}}}(d))$ is the set of degree $d$ smooth plane curves over ${\bb R}$. Let $d \in \bb Z_{\geq 0}$, $N = {d+2 \choose 2}$ and let $\mr B({\bb C}) \subset H^0(\bb P^2_{\bb C}, \OO_{\bb P^2_{\bb C}}(d)) \cong \bb C^N$ be the Zariski open subset of non-zero degree $d$ homogeneous polynomials $F$ that define smooth plane curves $\{F = 0 \} \subset \bb P^2({\bb C})$. Consider the universal plane curve $\mr B({\bb C}) \times \bb P^2({\bb C}) \supset \mr S(\CC) \to \mr B(\CC)$. The complex vector space $H^0(\bb P^2_{\bb C}, \OO_{\bb P^2_{\bb C}}(d))$ has a real structure, i.e. $\Gal({\bb C} / {\bb R})$ acts anti-linearly on it and this action preserves the space $\mr B({\bb C})$. The induced action on $\mr B({\bb C}) \times \bb P^2({\bb C})$ preserves in turn $\mr S({\bb C})$, and the morphism $\pi : \mr S({\bb C}) \to \mr B({\bb C})$ is Galois equivariant. Note that the family $\pi$ is algebraic, that is, comes from a morphism of algebraic varieties $\mu: \mr S \to \mr B$ over $\bb C$. By the above, $\mr B$, $\mr S$ and $\mu$ are actually defined over $\bb R$. For a projective and flat morphism of locally Noetherian schemes with integral geometric fibers, the relative Picard scheme exists \cite[\S V, 3.1]{FGA}. We obtain an abelian scheme 
$
\tn{\underline{Pic}}_{\mr S/\mr B}^0 \to \mr B
$ over $\bb R$ of relative dimension $g = (d-1)(d-2)/2$. By Theorem \ref{theorem1}, it suffices to satisfy Density Condition \ref{criterion}. In other words, we need to show the existence of $t \in \mr B({\bb C})$ and a non-zero $v \in H^{1,0}( \tn{Jac}(\mr S_t({\bb C})))$  such that $\langle v \rangle \otimes H^{1,0}(\tn{Jac}(\mr S_t({\bb C}))) \to T_t^*\mr B({\bb C})$ is injective. This is done in \cite[\textit{Proof of Proposition (6)}]{Colombo1990}, where Colombo and Pirola prove the complex analogue of Theorem \ref{theorem2}.\hyperlink{theoremC}{C}: an element $t \in \mr B(\CC)$ that satisfies the criterion is the $t \in \mr B(\CC)$ that corresponds to the Fermat equation $F = X_0^d + X_1^d + X_2^d$ (compare \cite[Proposition 3]{kim}).
\end{proof}

%% file: intermezzo.tex
Consider the set $\mr M_\RR$ (resp. $\mr M_\CC$) of isomorphism classes of a certain category of algebraic varieties over $\RR$ (resp. $\CC$) such that scalar extension defines a map $\mr M_\RR \to \mr M_\CC$ (think of polarized abelian varieties of fixed dimension, algebraic curves of fixed genus, polarized K3-surfaces, etc.). Suppose that $\mr M_{\bb C}$ admits a natural topology or even the structure of a complex analytic space. Can we define a reasonable topology on the set $\mr M_{\bb R}$ such that $\mr M_{\bb R} \to \mr M_{\bb C}$ is continuous? More precisely, if $\ca M$ is an algebraic moduli stack of finite type over $\bb R$ with coarse moduli space $\ca M \to M$, does $|\ca M(\bb R)|$ admit a natural topology making $|\ca M(\bb R)| \to M(\bb C)$ continuous? Contrary to the complex setting, $|\ca M(\bb R)| \to M(\bb R)$ need not be bijective, so that we cannot use the analytic topology on $M(\bb R)$. In this section we consider an algebraic stack $\mr X$ of finite type over $\RR$, pick an $\bb R$-surjective smooth presentation $\phi: X \to \mr X$, endow the set $|\mr X(\bb R)|$ with the quotient topology $\tau$ induced by $X(\bb R) \to |\mr X(\bb R)|$, and prove that $\tau$ does not depend on $\phi$. By Theorem \ref{rfonto2}, if $\mr X$ is Deligne-Mumford, one may use an $\bb R$-surjective \'etale presentation $\phi$ to define $\tau$.

\subsection{Pointwise surjective presentations of Deligne-Mumford stacks} \label{intermezzo}

\begin{notation} \label{notnot}
In this subsection, we fix a field $F$ which is either real closed or finite. For a scheme $S$, an algebraic stack $\mr X$ over $S$ and a scheme $T$ over $S$, we use the notation $|\mr X(T)|$ to denote the set of isomorphism classes of the groupoid $\mr X(T)$. 
\end{notation}

\noindent
The goal of Section \ref{intermezzo} is to prove that every Deligne-Mumford stack $\mr X$ of finite type over $F$ admits an \'etale presentation by an $F$-scheme $X$ such that the induced map $X(F) \to |\mr X(F)|$ is surjective. See Definition \ref{Rontodefinition} \& Theorem \ref{rfonto2} below. This statement is the \'etale analogue of Theorem A in \cite{2019}. In this article, Aizenbud and Avni prove that any algebraic stack $\mr X$ of finite type over a noetherian scheme $S$ admits a smooth presentation $\phi: X \to \mr X$ by an $S$-scheme $X$ such that for every morphism $\Spec F \to S$, the map $X(F) \to |\mr X(F)|$ is surjective. We use their result to extend it in the following way: if $\mr X$ is Deligne-Mumford over $S  = \Spec F$, then $\phi$ can be chosen \'etale.

\begin{definition}[\cite{Sakellaridis2016TheSS}]\label{Rontodefinition}
Let $\mr X_{/F}$ be an algebraic stack. A smooth presentation $X \to \mr X$ by an $F$-scheme $X$ is \textit{$F$-surjective} if the map $X(F) \to |\mr X(F)|$ is surjective. 
\end{definition}

\begin{theorem}[\cite{2019}] \label{smoothonto}
Any algebraic stack $\mr X$ of finite type over $F$ admits an $F$-surjective smooth presentation $\phi: X \to \mr X$ by a scheme $X$ over $F$. $\hfill \qed$
\end{theorem}

\begin{theorem} \label{rfonto2}
Any Deligne-Mumford stack $\mr X$ of finite type over $F$ admits an $F$-surjective \'etale presentation $\phi: X \to \mr X$ by a scheme $X$ over $F$. 
\end{theorem}

\begin{proof}
By Theorem \ref{smoothonto}, there exists an $F$-scheme $X$ and a $F$-surjective smooth presentation $P: X \to \mr X$. Let $y: \Spec F \to \mr X$ be any $F$-point of $\mr X$. Since $P$ is $F$-surjective, there exists an $F$-point $x: \Spec F \to X$ such that $P \circ x \cong y$. We claim that there exists a subscheme $j: U \hookrightarrow X$ such that $P|_U: U \to \mr X$ is \'etale, and such that $x: \Spec F \to X$ factors through an $F$-point $u: \Spec F \to U$. It will follow that $P|_U \circ u = P \circ j \circ u = P \circ x \cong y$. By taking the disjoint union $Y$ of all such subschemes $U \subset X$ we thus construct a scheme $Y$ over $F$ together with an $F$-surjective \'etale presentation $Y \to \mr X$. In other words, the claim implies the theorem. So let us prove this claim. The first part of its proof follows closely the proof of Theorem 8.1 in \cite{LM-B}; we shall shortly repeat those arguments for the convenience of the reader. Let $Z = X \times_{P, \mr X, P} X$ with $p_1: Z \to X$ and $p_2: Z \to X$ the two projections. Now $\Delta: \mr X \to \mr X \times_{F} \mr X$ is unramified \cite[Lemme 4.2]{LM-B} hence $p: Z \to X \times_{F} X$ is unramified. Consequently, $
p_1^*\Omega^1_{X/ F} \oplus p_2^*\Omega^1_{X/ F} =  p^*\Omega^1_{X \times_{F} X / F} \to \Omega^1_{Z / F} 
$ is surjective \cite[\S IV, 4, 17.2.2]{EGA}, hence by \cite[(8.2.3.2)]{LM-B}, the natural morphism of quasi-coherent $\ca O_X$-modules 
$
\Omega^1_{X/F} \to \Omega^1_{X / \mr X}
$ 
is surjective and $\Omega^1_{X/\mr X}$ is an $\ca O_X$-module locally free of finite rank. Let $r$ be the rank of $\Omega^1_{X/\mr X}$ around the point $x \in X(F)$. Because $\Omega^1_{X/F} \to \Omega^1_{X / \mr X}$ is surjective, there exists global sections $f_1, \dotsc, f_r$ of $\ca O_X$ of which the differentials at the point $x$ form a basis of the $k(x)$-vector space $\Omega^1_{X/\mr X} \otimes k(x)$. We obtain $F$-morphisms 
$$
f : = (f_1, \dotsc, f_r): X \to \bb A^r_{F}, \white (P,f): X \to \mr X \times_{F} \bb A^r_{F}.
$$
Then $(P,f)$ is a map of smooth algebraic stacks over $\mr X$ whose differential is an isomorphism at the point $x \in X(F)$, hence $(P,f)$ is \'etale at an neighbourhood $X' \subset X$ of $x $ \cite[\S IV, 4, 17.11.2]{EGA}. Replace $X$ by $X'$ and define $j: U \hookrightarrow X$ by the cartesian diagram
$$
\xymatrixcolsep{5pc}
\xymatrix{
U \ar[r] \ar[d]^j & \mr X \ar[d] & \mr X \times_{F} \Spec F \ar[dl]^{\id \times f(x)} \ar@{=}[l]\\
X \ar[r]^{(P,f) } & \mr X \times_{F} \bb A^r_{F}.  & }
$$
Then $P|_U: U \to \mr X$ is \'etale since $(P,f)$ is \'etale. Moreover, the morphisms $x: \Spec F \to X$ and $y: \Spec F \to \mr X$ are such that their images in $(\mr X \times_{F} \bb A^r_{F})(F)$ are given by 
$$(P \circ x,  f(x), \id ) \in (\mr X \times_{F} \bb A^r_{F})(F) \white \tn{ and } \white (y, f(x), \id) \in (\mr X \times_{F} \bb A^r_{F})(F).$$
But clearly the isomorphism between $P \circ x$ and $y$ in $\mr X(F)$ induces an isomorphism between $(P \circ x,  f(x), \id )$ and $(y, f(x), \id) $ in $(\mr X \times_{F} \bb A^r_{F})(F)$. By the universal property of the $2$-fibre product, $x$ and $y$ induce the required map $
u = (x,y): \Spec F \to U$. 
\end{proof}

\subsection{Topology on the real locus of a real algebraic stack}

\subfile{topologyrealpoints}

%% file: topologyrealpoints.tex
\begin{definition}\label{deffer}
Let $\mr X_{/ \bb R}$ be an algebraic stack of finite type. If $U \to \mr X$ is an $\bb R$-surjective smooth presentation (see Definition \ref{Rontodefinition}), the \textit{real analytic topology} on $|\mr X(\bb R)|$ is the quotient topology on $|\mr X(\bb R)|$ induced by the real analytic topology on $U(\bb R)$. 
\end{definition}
\begin{proposition} \label{th1}
The real analytic topology on $|\mr X(\bb R)|$ does not depend on the choice of $\bb R$-surjective presentation $U \to \mr X$. 
\end{proposition}
\begin{proof}
Consider two $\bb R$-surjective presentations $U \to \mr X$ and $V \to \mr X$. The stack $W: = V \times_{\mr X} U$ is an $\bb R$-scheme of finite type, and $\pr_U: W \to U$ and $ \pr_V: W \to V$ are smooth. We claim that $|\pr_U |: W(\bb R) \to U(\bb R)$ and $|\pr_V |: W(\bb R) \to V(\bb R)$ are surjective. Indeed, if $x \in U(\bb R)$ then $\pi(x) \in \tn{Ob}(\mr X(\bb R))$ which, by $\bb R$-surjectivity, implies that there is a $y \in V(\bb R)$ and an isomorphism $\psi: \eta(y) \cong \pi(x)$ in $\mr X(\bb R)$. Then $(x,y, \psi) \in \tn{Ob}(W(\bb R))$ and $|\pr_U|(x,y, \psi) = x$. Furthermore, $|\pr_U|$ and $|\pr_V|$ are open by Lemma \ref{2} below.
\end{proof}

\begin{lemma} \label{2}
Let $g: X \to Y$ be a morphism of schemes which are locally of finite type over $\bb R$. Let $|g|: X(\bb R) \to Y(\bb R)$ be the induced map of real analytic spaces. The morphism of analytic spaces $|g|$ is open if the morphism of schemes $g$ is smooth. 
\end{lemma}
\begin{proof}
We can work locally on $X(\bb R)$; let $U \subset X$ and $V \subset Y$ be affine open subschemes with $g(U)  \subset V$ such that $g|_U = \pr_2 \circ \pi$, where $\pi: U \to \bb A^d_V$ is \'etale for some integer $d \geq 0$, and $\pr_2: \bb A^d_V \to V$ is the projection on the right factor. Then $U(\bb R)$ is open in $X(\bb R)$ because $X(\bb C) \to X$ is a morphism of ringed spaces \cite[\S XII, 1.1]{SGA1}. % hence $X(\bb R) \to X(\bb C) \to X$ is continuous. 
It suffices to prove $U(\bb R) \to V(\bb R)$ is open, but this map factors as the composition $U(\bb R) \xrightarrow{|\pi |} \bb R^d \times V(\bb R) \xrightarrow{|\pr_2|} V(\bb R)$, where $|\pr_2|$ is open and $|\pi |$ a local homeomorphism. 
\end{proof}

\begin{corollary}\label{DMcorollary}
If $\mr X$ is a Deligne-Mumford stack of finite type over $\RR$, the topology induced on $|\mr X(\RR)|$ by any $\RR$-surjective \'etale presentation as in Theorem \ref{rfonto2} coincides
with the real-analytic topology. $\hfill \qed$
\end{corollary}
\begin{remark}
The assignment of a topological space to an algebraic stack of finite type over the real numbers is functorial. Moreover, if the real algebraic stack $\mr X$ admits a coarse moduli space $\pi: \mr X \to M$, then the natural map $|\mr X(\bb R)| \to M(\CC)$ is continuous. 
\end{remark}

%% file: comparingmodulispaces.tex
Let $\ca A_g$ be the algebraic stack over $\RR$ of principally polarized abelian varieties of dimension $g$, and let $\ca M_g$ be the stack over $\RR$ of smooth, proper and geometrically connected curves of genus $g$. Consider the real analytic topologies on $|\ca A_g(\RR)|$ and $|\ca M_g(\RR)|$, see Definition \ref{deffer}. Our goal is to prove that $|\ca A_g(\RR)|$ is homeomorphic to $\mr A_g^{\bb R}$, and $|\ca M_g(\RR)|$ to $\mr M_g^{\bb R}$, where the topology on $\mr A_g^{\bb R}$ (resp. $\mr M_g^{\bb R}$) is given by Equation (\ref{eq:topologyaggrossharris}) (resp. Equation (\ref{eq:topologymgseppalasilhol})).

\begin{theorem} \label{th:homeomorphismmoduli}
The natural bijection $|\ca A_g(\bb R)| \to \mr A_g^{\bb R}$ is a homeomorphism.
\end{theorem}

\begin{proof}
Let $\ca S \to \ca A_{g}$ be an $\bb R$-surjective \'etale surjection by a scheme $\ca S$ over $\bb R$ (see Theorem \ref{rfonto2}). Recall that $\ca A_g$ is smooth \cite[Remark 1.2.5]{Jong1993}, so that $\ca S(\bb C)$ is a complex manifold. Then $\ca S(\bb R) \to |\ca A_g(\bb R)| = \mr A_g^{\bb R} \cong \bigsqcup_{i \in I} \Gamma_i \setminus \bb H_g^{\tau_i}$ defines a surjective \textit{real period map}
\begin{equation} \label{realperiodmap}
\mr P: \ca S(\bb R) \longrightarrow \bigsqcup_{i \in I} \Gamma_i \setminus \bb H_g^{\tau_i}.
\end{equation}
We claim that $\mr P$ is continuous and open; this claim will finish the proof. 
\\
\\
In order to prove this we may work locally on $\ca S(\RR)$. So fix a point $0 \in \ca S(\RR)$, define $G = \Gal(\CC/\RR)$ as before and consider a $G$-stable contractible open neighbourhood $B$ of $0$ in $\ca S(\bb C)$ such that $B(\bb R) = B \cap \ca S(\bb R)$ is connected. The universal family $\ca X_g \to \ca A_g$ induces a holomorphic family of complex abelian varieties $\phi: A \to B$ with real structure $\tau: A \to A$, $\sigma: B \to B$, polarized by some section $E \in R^2\phi_\ast\ZZ$.  
\\
\\
Consider the real abelian variety $(A_0, \tau: A_0 \to A_0)$. Define $\Lambda_0$ to be the lattice $H_1(A_0, \ZZ)$ and consider the alternating form $E_0: \Lambda_0 \times \Lambda_0 \to \bb Z$. By \cite[Section 9]{grossharris}, there exists a symplectic basis $\underline m = \{m_1, \dotsc, m_g; n_1, \dotsc, n_g \}$ for $\Lambda_0$ and a unique element $i \in I$ such that the matrix corresponding to $\tau_*$ with respect to $\underline m$ is the matrix
\begin{equation} \label{eq:T}
    T = \begin{pmatrix}
 I_g & M\\
 0 & -I_g
\end{pmatrix} \in \GL_{2g}(\ZZ),
\end{equation}
where $M \in M_g(\ZZ)$ corresponds to $i \in I$ as in \cite[Ch.IV, Definition 4.4]{silholsurfaces} (see Section \ref{sec:densityinag}). 
\\
\\
The canonical trivialization $R^1\phi_*\bb Z \cong H^1(A_0, \bb Z)$ is $G$-equivariant, and induces
\begin{equation}\label{trivialization2}
r: R^1\phi_*\bb Z\cong H^{1}(A_0, \bb Z) \xrightarrow{E_0} H_1(A_0, \bb Z) = \Lambda_0, \tn{ inducing } r_t: H_1(A_t, \ZZ) \cong \Lambda_0 \; \forall t \in B.
\end{equation}
Consequently, for every $t \in B$, the lattice $H_1(A_t, \bb Z)$ is equipped with a symplectic basis $r_t^{-1}(\underline m) = \{m(t)_1, \dotsc, m(t)_g; n(t)_1, \dotsc, n(t)_g\}$. We claim that these bases are compatible with the real structure of $\phi: A \to B$ in the sense that for all $t \in B$, the pushforward
$$
\tau_*: H_1(A_t, \bb Z) \to H_1(A_{\sigma(t)}, \bb Z)
$$
is given by matrix $T$ of Equation (\ref{eq:T}) with respect to the bases $r_t^{-1}(\underline m)$ and $r_{\sigma(t)}^{-1}(\underline m)$. Indeed, this follows from the fact that the following diagram commutes:
\begin{equation} \label{comofcom}
\xymatrixcolsep{2pc}
\xymatrix{
\Lambda_0\ar[d]^{\tau_*} &H^1(A_0, \bb Z)\ar[d]^{-\tau*} \ar[l]&H^1(A_t, \bb Z) \ar[d]^{-\tau_*}\ar[l] & H_1(A_t, \bb Z) \ar[l]\ar[d]^{\tau_*} \\ 
\Lambda_0 &H^1(A_0, \bb Z)\ar[l] &H^1(A_{\sigma(t)}, \bb Z) \ar[l] & H_1(A_{\sigma(t)}, \bb Z) \ar[l] 
}
\end{equation}
The first and the third square commute because if $\tau: A_t \to A_{\sigma(t)}$ is the anti-holomorphic restriction of $\tau$ to the fibre $A_t$, then $E_t(x,y) = - E_{\sigma(t)}(\tau_\ast(x), \tau_\ast(y))$ for all $x,y \in H_1(A_t, \ZZ)$. The second diagram commutes because the trivialization of $R^1\phi_*\bb Z$ is $G$-equivariant. 
\\
\\
Now the family of symplectic bases $\{r_t^{-1}(\underline m)\subset H_1(A_t, \bb Z) \}_{t \in B}$ induces a holomorphic period map $P: B \to \bb H_g$, which is defined by sending an element $t \in B$ to the period matrix of the abelian variety $A_t$ with respect to the symplectic basis $r_t^{-1}(\underline m)$. Note that the differential $dP$ defines an isomorphism on each tangent space since $\ca S \to \ca A_g$ is \'etale. 
\\
\\
Consider the anti-holomorphic involution $\tau_i: \bb H_g \to \bb H_g$ defined by $\tau_i(Z) = M - \bar Z$, see Section \ref{sec:densityinag}. We claim that $P$ is Galois-equivariant, in the sense that $\tau_i \circ P = P \circ \sigma$. Assuming this claim for the moment, we obtain an induced morphism of smooth manifolds $P_\RR: B(\RR) = B^\sigma \to \bb H_g^{\tau_i}$, whose differential is an isomorphism on each tangent space because the same holds for $P$. In particular, $P_\RR$ is open, and hence the composition 
\begin{equation} \label{eq:composition}
    B(\RR) \xrightarrow{P_\RR} \bb H_g^{\tau_i} \to \Gamma_i \setminus \bb H_g^{\tau_i}
\end{equation}
is open as well. But (\ref{eq:composition}) equals the restriction of $\mr P$ to the small open subset $B(\RR) \subset \ca S(\RR)$. Therefore, $\mr P$ is continuous and open at every point of $\ca S(\RR)$, and we are done. 
\\
\\
So we are reduced to the claim that, for all $t \in B$, one has $\tau_i\left(P(t)\right) = P(\sigma(t))$. Consider the symplectic basis $r_t^{-1}(\underline m) = \{m(t)_i; n(t)_j\}_{ij}
\subset H_1(A_t, \ZZ)$. Let $\{\omega(t)_1, \dotsc, \omega(t)_g\} \subset H^0(A_t, \Omega_{A_t}^1)$ be the basis dual to $\{m(t)_1, \dotsc, m(t)_g\}$ under the natural pairing
\begin{equation} \label{eq:pairing}
    H_1(A_t, \ZZ) \times H^0(A_t, \Omega_{A_t}^1) \to \CC,\white (\gamma, \omega) \mapsto \int_\gamma\omega.
\end{equation}
Then we have 
$
P(t)= \left(\int_{n(t)_i} \omega(t)_j\right)_{ij} \in \bb H_g. 
$
On the other hand, the following diagram commutes by a generalization of Lemma \ref{lemmatje}: 
\begin{equation} \label{comofcom}
\xymatrixcolsep{2pc}
\xymatrix{
 H_1(A_t, \bb Z) \ar[d]^{\tau_*}  \ar[r] & H^0(A_t, \Omega^1_{A_t})^\vee \ar[d]^{F_{dR}} \\
H_1(A_{\sigma(t)}, \bb Z)  \ar[r]& H^0(A_{\sigma(t)}, \Omega^1_{A_{\sigma(t)}})^\vee,
}
\end{equation}
where $F_{dR}$ is the anti-linear map induced by the differential $d\tau: T_eA_t \to T_eA_{\sigma(t)}$ of $\tau: A_t \to A_{\sigma(t)}$ and the identification $T_eA = H^0(A, \Omega^1_{A})^\vee$ for any abelian variety $A$. This implies that (\ref{eq:pairing}) is compatible with $\tau_\ast$ and $F_{dR}$, and $F_{dR}\left(\omega(t)_j \right) = \omega(\sigma(t))_j$. Therefore,
$$
\tau_j \left( P(t) \right) = M - \overline{\left(\int_{n(t)_i} \omega(t)_j\right)_{ij}} = M - \left(\int_{\tau_\ast(n(t)_i)} F_{dR}\left(\omega(t)_j \right)\right)_{ij} = M - 
\left(\int_{\tau_\ast(n(t)_i)} \omega(\sigma(t))_j\right)_{ij}. 
$$
If $M = \left( x_{ij} \right)_{ij}$, and $
\phi_t: H_1(A_t, \ZZ) \cong \ZZ^{2g}
$ is the isomorphism identifying $r_t^{-1}(\underline m)$ with the standard symplectic basis $\{ e_1, \dotsc, e_g ; f_1, \dotsc , f_g\} \subset \ZZ^{2g}$, then, for $T \in \GL_{2g}(\ZZ)$ as in (\ref{eq:T}):
$$
\tau_\ast(n(t)_i) = \phi_{\sigma(t)}^{-1}\left( T\cdot f_i \right) = \phi_{\sigma(t)}^{-1}\left(  \sum_{k = 1}^g x_{ki} e_k - f_i \right) = \sum_{k = 1}^g x_{ki} m(\sigma(t))_k  - n(\sigma(t))_i \in H_1(A_{\sigma(t)}, \ZZ). 
$$
Therefore, we get that $\int_{\tau_\ast(n(t)_i)} \omega(\sigma(t))_j = \int_{\sum_{k = 1}^g x_{ki} m(\sigma(t))_k  - n(\sigma(t))_i} \omega(\sigma(t))_j$, hence 
$$
\left(\int_{\tau_\ast(n(t)_i)} \omega(\sigma(t))_j\right)_{ij} = 
\left(\int_{x_{ji} m(\sigma(t))_j} \omega(\sigma(t))_j\right)_{ij} - 
\left(\int_{n(\sigma(t))_i} \omega(\sigma(t))_j\right)_{ij} = (x_{ji})_{ij} - P\left(\sigma(t)\right). 
$$
But $\left( x_{ji}\right)_{ij} = \left( x_{ij}\right)_{ij} = M$ for $M$ is symmetric, whence our claim $\tau_i\left(P(t)\right) = P(\sigma(t))$.
\end{proof}

\begin{theorem} \label{th:homeomorphismmoduli2}
The natural bijection $|\ca M_g(\bb R)| \to \mr M_g^{\bb R}$ is a homeomorphism. 
\end{theorem}

\begin{proof}
We proceed as in the proof of Theorem \ref{th:homeomorphismmoduli}. Let $\ca S \to \ca M_{g}$ be an $\bb R$-surjective \'etale presentation by a scheme $\ca S$ over $\bb R$ corresponding to a family of genus $g$ curves $\ca C \to \ca S$ over $\bb R$. The composition $\mr P: \ca S(\bb R) \to |\ca M_g(\bb R)| = \mr M_g^{\bb R} \cong \bigsqcup_{j \in J} N_j \setminus \ca T_g^{\sigma_j}$ is surjective. The goal is to prove that $\mr P$ is continuous and open, which will imply the result. 
 \\
 \\
Note that $\ca S(\CC)$ is a complex manifold because ${\ca M_g}$ is smooth \cite[Theorem 5.2]{DM69}. Fix $0 \in \ca S(\RR)$ and consider a $G$-stable contractible open neighborhood $B \subset \ca S(\CC)$ of $0$ such that $B(\RR) = B \cap \ca S(\RR)$ is connected. Let $\phi: C \to B$ be the induced family of Riemann surfaces over $B$ with real structure $\sigma: B \to B$, $\tau: C \to C$. By \cite{seppalasilhol2}, there exists a Teichm\"uller structure $[f_0], f_0: C_0 \xrightarrow{\sim} \Sigma$ and a unique $j \in J$ such that the composition
$$
\Sigma \xrightarrow{f_0^{-1}} C_0 \xrightarrow{\tau} C_0 \xrightarrow{f_0} \Sigma 
$$
is isotopic to the map $\sigma_j: \Sigma \to \Sigma$ - in other words, such that $[f_0 \circ \tau \circ f_0^{-1} ] = [\sigma_j]$. 
\\
\\
Since the Kodaira-Spencer morphism $\rho: T_0B \to H^1(C_0, T_{C_0})$ is an isomorphism, the family $\phi: C \to B$ is a Kuranishi family for the fiber $C_0$. Since $B$ is contractible, $\phi: C \to B$ is topologically trivial hence can be endowed with a unique Teichm\"uller structure $\{[f_t], f_t: C_t \to \Sigma\}_{t \in B}$ extending the Teichm\"uller structure $[f_0]$ on $C_0$. For $t \in B$, consider the anti-holomorphic map $\tau: C_t \to C_{\sigma(t)}$. We claim that the two Teichm\"uller structures $[f_t]$ and $[\sigma_{j} \circ f_{\sigma(t)} \circ \tau]$ on $C_t$ agree. Indeed, the family $\left\{ [\sigma_{j} \circ f_{\sigma(t)} \circ \tau] \right\}_{t \in B}$ defines a new Teichm\"uller structure on $\phi: C \to B$, agreeing with $\left\{ [f_t] \right\}_{t \in B}$ at the point $0 \in B$, and therefore agreeing with $\left\{ [f_t] \right\}_{t \in B}$ everywhere on $B$ by \cite[XV, \S 2]{curvesII}.
\\
\\
Consequently, the holomorphic map $P: B \to \ca T_g$ defined as $P(t) = [C_t, [f_t]]$ can be shown to be $G$-equivariant as follows. The involution $\sigma_j: \ca T_g \to \ca T_g$ is defined by sending the class $[C, [f]]$ of a curve $C$ with Teichm\"uller structure $[f]$ to $[C^\sigma, [C^\sigma \xrightarrow{\text{can}} C \xrightarrow{f} \Sigma \xrightarrow{\sigma_j}\Sigma ]] \in \ca T_g$, where $C^\sigma$ is the complex conjugate of $C$ and $\text{can}: C^\sigma \to C$ is the canonical anti-holomorphic map. The family $\phi^\sigma: C^\sigma \to B^\sigma$ is isomorphic to the family $\phi: C \to B$ via the maps $\sigma \circ \text{can}: B^\sigma \to B$ and $\tau \circ \text{can}: C^\sigma \to C$, and the composition $\tau \circ \text{can}: C^\sigma \cong C$ restricts to an isomorphism $C_t^\sigma \cong C_{\sigma(t)}$ for each $t \in B$. Together this implies that 
$
\sigma_j \left( P(t)  \right) = \sigma_j \left( [C_t, [f_t]]  \right) = [C_{\sigma(t)}, [C_{\sigma(t)} \xrightarrow{\tau} C_t \xrightarrow{f_t} \Sigma \xrightarrow{\sigma_j} \Sigma]] = [C_{\sigma(t)}, [f_{\sigma(t)}]] = P(\sigma(t))$. So indeed, $\sigma_j \circ P = P \circ \sigma$. Let $P_\RR : B(\RR) \to \ca T_g^{\sigma_j}$ be the induced morphism on real loci. Since $P$ induces isomorphisms on tangent spaces, the same holds for $P_\RR$. Consider the projection $\pi: \ca T_g^{\sigma_j} \to N_j \setminus \ca T_g^{\sigma_j}$. Then $\pi \circ P_\RR = \mr P|_{B(\RR)}: B(\RR) \to N_j \setminus \ca T_g^{\sigma_j}$, which implies that $\mr P$ is continuous and open at the point $0 \in \ca S(\RR)$, and we are done. 
\end{proof}